\documentclass[12pt]{amsart}
\usepackage{fullpage,url}
\newtheorem{theorem}{Theorem}[section]
\newtheorem{lemma}[theorem]{Lemma}

\newtheorem{corollary}[theorem]{Corollary}
\theoremstyle{remark}
	\newtheorem{definition}[theorem]{Definition}
	\newtheorem{remark}[theorem]{Remark}
\everymath{\displaystyle}

\newcommand{\calH}{\mathcal{H}}
\newcommand{\calM}{\mathcal{M}}
\newcommand{\frakm}{\mathfrak{m}}
\newcommand{\frakn}{\mathfrak{n}}
\newcommand{\frako}{\mathfrak{o}}
\newcommand{\FF}{\mathbb{F}}
\newcommand{\QQ}{\mathbb{Q}}
\newcommand{\RR}{\mathbb{R}}
\newcommand{\ZZ}{\mathbb{Z}}

\DeclareMathOperator{\Frac}{Frac}
\DeclareMathOperator{\Spa}{Spa}
\DeclareMathOperator{\Spd}{Spd}
\DeclareMathOperator{\Ker}{Ker}

\begin{document}

\title{On the relative Nullstellensatz in nonarchimedean geometry}
\author{Kiran S. Kedlaya and Yutaro Mikami}


\begin{abstract}
We establish a relative version of the Nullstellensatz for algebras topologically of finite type over a given Banach Tate ring $A$, under the assumption that the corresponding statement holds for rational localizations of $A$. This applies in particular to pseudoaffinoid algebras and to the coordinate rings of affinoid subspaces of a Fargues--Fontaine curve.
\end{abstract}

\maketitle

The development of rigid analytic geometry over a nonarchimedean field $K$, as found for instance in \cite{bgr}, involves careful proof of some analogues of classical statements about rings of finite type over a field, including the Noether normalization theorem and the Nullstellensatz. We investigate the extent to which such statements can be made in a relative form when the field $K$ is replaced by a more general topological base ring.

We state our results in terms of commutative Banach rings as some of our proofs proceed most naturally in this language; we will only consider such rings that are \emph{Tate} in the sense that they contain a topologically nilpotent unit (e.g., Banach algebras over a nonarchimedean field). 
That said, any Huber ring which is Tate is the underlying topological ring of some commutative Banach Tate ring \cite[Remark~1.5.4]{aws}, so our results apply in that setting also.

Our main result (Theorem~\ref{T:Nullstellensatz property}) says that if $A$ is a commutative Banach Tate ring and the analogue of the Nullstellensatz holds for rational localizations of $A$, then it also holds for every algebra topologically of finite type over $A$. The conditional statement is needed to rule out certain pathologies (Remark~\ref{R:Banach field}), but the hypothesis is known to hold in certain situations such as for classical affinoid algebras over a nonarchimedean field, and for \emph{pseudoaffinoid} algebras over a discretely valued nonarchimedean field by the work of Louren\c{c}o  \cite{Lou17}. One new case is where $A$ is the coordinate ring of an affinoid subspace of a Fargues--Fontaine curve (Theorem~\ref{T:FF strong Null}).

In the following arguments, let $\calM(A)$ denote the Gelfand spectrum of a commutative Banach Tate ring $A$;
this is empty if and only if $A$ is the zero ring \cite[Theorem~1.2.1]{berkovich}. For $\alpha \in \calM(A)$, let $\calH(\alpha)$ denoted the completed residue field of $A$ at $\alpha$.
Let $A^{\circ}$ denote the subring of $A$ consisting of power-bounded elements, and let $A^{\circ\circ}$ denote the ideal of $A^{\circ}$ consisting of topologically nilpotent elements.

\section{Relative Weierstrass preparation and Noether normalization}

\begin{lemma}[Relative Weierstrass division]\label{L:division}
Let $A$ be a commutative Banach Tate ring.
Let $n_0\geq 0$ be an integer.
Let $f(T)=\sum_{n=0}^{\infty}f_nT^n \in A\langle T \rangle$ be a series such that:
\begin{enumerate}
\item $f_{n_0}=1$.
\item $|f_n|\leq 1$ for all $n<n_0$.
\item $|f_n|< 1$ for all $n>n_0$.
\end{enumerate}
Then for any $g\in A\langle T \rangle$, there exists a unique pair $(q,r)$ with $q\in A\langle T \rangle$, $r\in A[T]$, and $\deg r <n_0$ such that $g=fq+r$.
Moreover, in this case we have $|g|=\max \{|q|,|r|\}$, where we endow $A\langle T \rangle$ with the Gauss norm.
\end{lemma}
\begin{proof}
Let us prove $|fq+r|=\max \{|q|,|r|\}$ for $q=\sum_{n=0}^{\infty}q_nT^n\in A\langle T \rangle$ and $r\in A[T]$ with $\deg r <n_0$.
Let $m$ be the largest integer such that $|q_m|=|q|$.
Then, the norm of the coefficient of $T^{n_0+m}$ in $fq$ is equal to $|q_m|$.
Therefore, we get $|fq|\geq |q|$. 
Since the Gauss norm is submultiplicative, we get $|fq|= |q|$.
In particular, we have $|fq+r|\leq \max \{|q|,|r|\}$, and if $|q|<|r|$, then the equality holds.
If $|q|\geq |r|$, then the norm of the coefficient of $T^{n_0+m}$ in $fq+r$ (or equivalently, in $fq$) is equal to $|q_m|$.
Thus, we get $|fq+r|\geq |q| \geq \max \{|q|,|r|\}$, which implies $|fq+r|= \max \{|q|,|r|\}$.

The remaining claims can be proved in the same way as in \cite[Theorem~5.2.1/2]{bgr}.
\end{proof}

\begin{lemma}[Relative Weierstrass preparation] \label{L:Weierstrass prep}
Let $A$ be a commutative Banach Tate ring.
Let $f = \sum_{n=0}^\infty f_n T^n \in A \langle T \rangle$ be a series such that for some nonnegative integer $n_0$:
\begin{itemize}
\item
$f_{n_0}$ is a unit in $A$;
\item
$|f_{n_0}^{-1} f_n| \leq 1$ for all $n < n_0$;
\item
$|f_{n_0}^{-1} f_n| < 1$ for all $n > n_0$.
\end{itemize}
Then $f$ factors uniquely as $gu$ where $g \in A[T]$ is a monic polynomial of degree $n_0$ and $u \in A \langle T \rangle^\times$ is a unit.
\end{lemma}

\begin{proof}
We may assume $f_{n_0}=1$ by multiplying $f$ by $f_{n_0}^{-1}$.
By using Lemma~\ref{L:division}, we can prove the claim in the same way as in \cite[Theorem~5.2.2/1]{bgr}.
\end{proof}

\begin{lemma} \label{L:perturb finite presentation}
Let $A$ be a commutative Banach Tate ring. 
Let $A \langle X_1,\dots,X_m \rangle \to B$ be a finite morphism, and we write $x_i$ for the image of $X_i$ in $B$.
Then for any $t_1,\dots,t_m \in B$ with $|t_i-x_i|$ sufficiently small, 
the induced morphism $A[T_1,\dots,T_m] \to B;\; T_i\mapsto t_i$ extends to a finite morphism
$A \langle T_1,\dots,T_m \rangle  \to B$.
\end{lemma}
\begin{proof}
Let $A_0$ be a ring of definition of $A$.
By choosing a finite generating set of $B$ over $A \langle X_1,\dots,X_m \rangle$
and then rescaling, we construct a finite $A_0 \langle X_1,\dots,X_m \rangle$-subalgebra $R$ of $B$
such that $A \otimes_{A_0} R \cong B$. If we then fix a topologically nilpotent unit $\varpi \in A_0$
and restrict $t_1,\dots,t_m$ by specifying that $t_i - x_i \in \varpi R$,
then $R/\varpi$ is finite over $(A_0/\varpi)[X_1,\dots,X_m] = (A_0/\varpi)[T_1,\dots,T_m]$.
Since $R$ is $\varpi$-adically complete, we deduce from this that $R$ is finite over $A_0 \langle T_1,\dots,T_m \rangle$, and hence $B$ is finite over $A \langle T_1,\dots,T_m \rangle$.
\end{proof}

\begin{lemma}\label{L:local finite}
Let $R$ be a commutative Banach Tate ring, fix $\alpha \in \calM(R)$, and set $K := \calH(\alpha)$.
Let $A$, $B$ be commutative Banach $R$-algebras topologically of finite type, and let $A\to B$ be a morphism of commutative Banach $R$-algebras.
If the induced map $A \widehat{\otimes}_R K \to B \widehat{\otimes}_R K$ is finite, then there exists a rational localization $R \to R'$ with $\alpha \in \calM(R')$ such that the induced map $A \widehat{\otimes}_R R' \to B \widehat{\otimes}_R R'$ is finite. 
\end{lemma} 
\begin{proof}
Since the morphism $R \to B$ is topologically of finite type, so is $A \to B$;
we can thus choose $x_1,\ldots,x_n\in B^{\circ}$ such that the induced map $A \langle X_1,\dots,X_n \rangle \to B ;\; X_i\mapsto x_i$ is a quotient map.
Since $A \widehat{\otimes}_R K \to B \widehat{\otimes}_R K$  is finite and the image of $x_i$ in $B \widehat{\otimes}_R K$  lies in $(B \widehat{\otimes}_R K)^{\circ}$, we can find a monic polynomial $f_i(T)\in (A \widehat{\otimes}_R K)^{\circ}[T]$ such that $f_i(x_i)=0$ in $B \widehat{\otimes}_R K$.
Since the image of $\varinjlim_{R\to R'} A\widehat{\otimes}R'\to A\widehat{\otimes}K$ is dense in  $A\widehat{\otimes}K$, where the colimit is taken over all rational localizations $R\to R'$ such that $\alpha\in \calM(R')$, 
we may replace $R$ with a rational localization $R \to R'$ satisfying $\alpha \in \mathcal{M}(R')$ (and replace $A$ and $B$ with $A \widehat{\otimes}_R R'$ and $B \widehat{\otimes}_R R'$, respectively) 
so that there exists a monic polynomial $g_i(T) \in A^{\circ}[T]$ with 
$f_i(T) - g_i(T) \in (A \widehat{\otimes}_R K)^{\circ\circ}[T].$
Then $g_i(x_i)$ is topologically nilpotent in $B \widehat{\otimes}_R K$.
By \cite[Lemma~1.4.8]{Hub96}, there exists a rational localization $R\to R'$ with $\alpha \in \calM(R')$ such that $g_i(x_i)$ is topologically nilpotent in $B \widehat{\otimes}_R R'$.
Then, the induced map $A \widehat{\otimes}_R K \to B \widehat{\otimes}_R K$ is finite by \cite[Lemma~1.4.3]{Hub96}.
\end{proof}

\begin{theorem}[Relative Noether normalization] \label{L:noether normalization}
Let $A$ be a commutative Banach Tate ring, fix $\alpha \in \calM(A)$, and set $K := \calH(\alpha)$.
Let $B$ be a commutative Banach $A$-algebra topologically of finite type.
Let $m$ be the dimension of the affinoid algebra $B \widehat{\otimes}_A K$ over $K$.
Then for some rational localization $A \to A'$ with $\alpha \in \calM(A')$,
there exists a finite $A'$-linear morphism $A' \langle T_1,\dots,T_m \rangle \to B \widehat{\otimes}_A A'$ such that the induced map $K \langle T_1,\dots,T_m \rangle \to B \widehat{\otimes}_A K$ is injective.
\end{theorem}
\begin{proof}
By Lemma~\ref{L:local finite}, it is enough to show that for some choice of $A'$, there exist $t_1,\dots,t_m \in B \widehat{\otimes}_A A'$ for which we obtain an induced map
$A' \langle T_1,\dots,T_m \rangle \to B \widehat{\otimes}_A A';\; T_i\mapsto t_i$ such that the further induced map $K \langle T_1,\dots,T_m \rangle \to B \widehat{\otimes}_A K$ is finite and injective.
We apply the Noether normalization theorem for affinoid algebras given in
\cite[Corollary~6.1.2/2]{bgr} to get $t_1,\dots,t_m \in B \widehat{\otimes}_A K$
giving rise to a finite map $K \langle T_1,\dots,T_m \rangle \to B \widehat{\otimes}_A K;\; T_i\mapsto t_i$.
We then use Lemma~\ref{L:perturb finite presentation} to replace the $t_i$ with sufficiently close elements $t_i \in B \widehat{\otimes}_A A'$ for some $A'$, which gives what we want. (The resulting map must be injective because $B \widehat{\otimes}_A K$ is of dimension $m$ over $K$; keep in mind that this relies on properties of dimension that themselves hinge on \cite[Corollary~6.1.2/2]{bgr}.)
\end{proof}

\section{Relative Nullstellensatz}

The following is a variant of \cite[Proposition~2.14]{banach-field}.

\begin{lemma} \label{lem:dense field}
Let $R$ be a one-dimensional connected affinoid algebra over a nonarchimedean field $K$.
Then there exists an open subset $U$ of $R$ not containing $0$, none of whose elements is invertible.
Consequently, no dense subring of $R$ can be a field.
\end{lemma}
\begin{proof}
By \cite[Corollary~6.1.2/2]{bgr} (or Theorem~\ref{L:noether normalization}), 
there exists a finite injective morphism
$K \langle T \rangle \to R$. Let $I$ be the $K\langle T \rangle$-torsion submodule of $R$; then $I$ is an ideal of $R$ and $R/I$ is finite flat over $K \langle T \rangle$ of some rank $n$ (because the latter is a principal ideal domain). It will suffice to check the claim with $R$ replaced by $R/I$;
that is, we may assume hereafter that $R/I$ is finite flat over $K \langle T \rangle$ of some rank $n$.

At this point, the norm map $N\colon R \to K \langle T \rangle$ is continuous.
We may thus take $x = T - \lambda \in K \langle T \rangle$ where $\lambda \in K$ satisfies $0 < |\lambda| < 1$ and $U_0 = \{y \in K \langle T \rangle \colon |x-y| < |\lambda|\}$,
so that $U = N^{-1}(U_0)$ is an open subset of $R$ not containing $0$. 
For any $z \in U$, $N(z) \in U$ is not a unit in $K \langle T \rangle$
(by comparison of Newton polygons), and so $z$ is not a unit in $R$.
\end{proof}

\begin{lemma} \label{L:relative Nullstellensatz}
Let $A$ be a commutative Banach Tate ring. 
Let $A \langle T \rangle \to B_1$ be a rational localization. Let $\frakm_1$ be a maximal ideal of $B_1$. Then there exist a rational localization $A \to B$,
a maximal ideal $\frakm'_1$ of $B'_1 := B_1 \widehat{\otimes}_A B$ containing $\frakm_1 B'_1$, 
a maximal ideal $\frakm$ of $B$, and a finite $A$-linear morphism $B/\frakm \to B'_1/\frakm'_1$.
\end{lemma}
\begin{proof}
Since $\frakm_1$ is a maximal ideal, the quotient $B_1/\frakm_1$ is itself a nonzero commutative Banach ring;
its spectrum $\calM(B_1/\frakm_1)$ is therefore nonempty (but need not consist of a single point; see Remark~\ref{R:Banach field}).
Fix $\alpha \in \calM(A)$ in the image of the projection $\calM(B_1/\frakm_1) \to \calM(A)$;
over the course of the argument, we will be free to make the following changes to the input data.
\begin{itemize}
\item
Replacing $A$ with a rational localization $B$ such that $\calM(B)$ contains $\alpha$, at the same time replacing $B_1$ with $B_1 \widehat{\otimes}_A B$.
\item
Replacing $B_1$ with a ring $B_0$ such that $A \to B_1$ can be refactored as 
$A  \to A\langle T'\rangle \to B_0 \to B_1$ where the middle morphism is a rational localization, and the last morphism is finite,
at the same time replacing $T$ and $\frakm_1$ with $T'$ and $\frakm_1 \cap B_0$, respectively. (By the going up theorem,
$\frakm_1 \cap B_0$ is a maximal ideal of $B_0$.)
\end{itemize}

Let $U_\alpha$ be the locus of $\pi^{-1}(\alpha) \cong \calM(B_1 \widehat{\otimes}_A \calH(\alpha))$ cut out by $\frakm_1$.
As a nonempty Zariski closed subspace of a one-dimensional affinoid space over the nonarchimedean field $\calH(\alpha)$, $U_\alpha$ has a well-defined dimension which is 0 or 1, and which remains unchanged under any replacement of the input data as described above.

We will show by way of contradiction that the dimension of $U_\alpha$ cannot be 1.
Suppose the contrary; then $\frakm_1$ maps to zero in some connected component $R$ of $B_1 \widehat{\otimes}_A \calH(\alpha)$. This means that the field $B_1/\frakm_1$ maps to the dense subring of $R$ generated by $B_1 \otimes_A \Frac (A/(\frakm_1 \cap A))$, yielding a contradiction against Lemma~\ref{lem:dense field}.

By the previous paragraph, $U_\alpha$ is a zero-dimensional rigid analytic space over $\calH(\alpha)$.
By Theorem~\ref{L:noether normalization}, after replacing the input data we may reduce to the case $B_1 = A \langle T \rangle$.
(Note that this step is not necessary if we begin with $B_1 = A \langle T \rangle$, but that special case will not be sufficient for our later arguments.)
 Since $U_\alpha \neq \calM(\calH(\alpha) \langle T \rangle)$,
we can now find some $f \in \frakm_1$ which does not vanish identically on $\calM(\calH(\alpha) \langle T \rangle)$. 
After replacing the input data and rescaling, we may assume that $f$ is of the form as in Lemma~\ref{L:division}.
Then by Lemma~\ref{L:division}, $A\to A\langle T\rangle/(f)$ is finite, so the maximal ideal $\frakm_1 A \langle T \rangle/(f)$ contracts to a maximal ideal $\frakm$ of $A$ and the claim follows.
\end{proof}
\begin{definition}
We say that a commutative Banach Tate ring $A$ has the \emph{weak Nullstellensatz property} if 
for every rational localization $A \to B$
and every maximal ideal $\frakm$ of $B$, the ring homomorphism $A/(\frakm \cap A) \to B/\frakm$ is finite. This implies in particular that $\frakm \cap A$ is maximal in $A$.

If in addition for every maximal ideal $\frakn$ of $A$, $A/\frakn$ always carries the topology defined by some multiplicative norm, we say that $A$ has the \emph{strong Nullstellensatz property}. 
In this case, for every maximal ideal $\frakm$ of $B$, $A/(\frakm \cap A) \to B/\frakm$ is an isomorphism\footnote{Since $A/(\frakm \cap A) \to B/(\frakm \cap A)B\neq 0$ is a rational localization of the nonarchimedean field $A/(\frakm \cap A)$, $A/(\frakm \cap A) \to B/(\frakm \cap A)B$ is an isomorphism and $B/(\frakm \cap A)B$ is also a nonarchimedean field. In particular, $\frakm=(\frakm \cap A)B$ and $A/(\frakm \cap A) \to B/\frakm$ is an isomorphism. We note that if $A/(\frakm \cap A)$ is merely a Banach field, then $A/(\frakm \cap A) \to B/(\frakm \cap A)B\neq 0$ is not necessarily an isomorphism.}.
In particular, $B/\frakm$ is also a nonarchimedean field.
\end{definition}

\begin{remark} \label{R:Banach field}
For $A$ a commutative Banach Tate ring, it is possible for the underlying ring of $A$ to be a field 
but for the topology of $A$ not to be defined by a multiplicative norm; this phenomenon is discussed at length in \cite{banach-field}. 
Such a ring $A$ does not have the strong Nullstellensatz property.

However, all of the examples considered in \cite{banach-field} have the property that $\calM(A)$ is a single point. This means that any rational localization $A \to B$ is an isomorphism, so $A$ trivially has the weak Nullstellensatz property. 
In fact, we do not have an example for which the weak Nullstellensatz property fails, though it seems that this should be possible.
\end{remark}

\begin{lemma} \label{L:induction on Null}
Let $A$ be a commutative Banach Tate ring with the weak Nullstellensatz property.
For $\phi\colon A \langle T \rangle \to B_1$ a rational localization and $\frakm_1$ a maximal ideal of $B_1$, the induced map $A/(\phi^{-1}(\frakm_1) \cap A) \to B_1/\frakm_1$ is finite.
\end{lemma}
\begin{proof}
Set notation as in Lemma~\ref{L:relative Nullstellensatz}; then
\[
A/(\frakm \cap A) \to B/\frakm \to B'_1/\frakm'_1
\]
is a finite morphism; in particular, $\frakm \cap A$ is a maximal ideal of $A$.
By refactoring through $B_1/\frakm_1$, we see that $A/(\frakm \cap A) \to B_1/\frakm_1$ is finite;
in particular, $\frakm \cap A = \phi^{-1}(\frakm_1) \cap A$.
\end{proof}

\begin{theorem}[Relative Nullstellensatz] \label{T:Nullstellensatz property}
Let $A$ be a commutative Banach Tate ring with the weak (resp.\ strong) Nullstellensatz property.
Then for every nonnegative integer $n$, the following statements hold.
\begin{enumerate}
\item[(a)]
The ring $A \langle T_1,\dots,T_n \rangle$ has the weak (resp.\ strong) Nullstellensatz property.
\item[(b)]
For every finite morphism $\phi\colon A \langle T_1,\dots,T_n \rangle \to B_n$ and every maximal ideal $\frakm_n$ of $B_n$, the ring homomorphism $A/(\phi^{-1}(\frakm_n) \cap A) \to B_n/\frakm_n$ is finite.
\end{enumerate}
\end{theorem}
\begin{proof}
To deduce (a), it suffices to check the case $n=1$.
In this case, writing $T$ for $T_1$, with notation as in Lemma~\ref{L:induction on Null},
the finite map $A/(\phi^{-1}(\frakm_1) \cap A) \to B_1/\frakm_1$  factors through
$A \langle T \rangle/\phi^{-1}(\frakm_1) \to B_1/\frakm_1$, so the latter is finite.
This confirms that $A \langle T \rangle$ has the weak (resp. strong) Nullstellensatz property.

To deduce (b), by the going up theorem we formally reduce to the case $B = A \langle T_1,\dots,T_n \rangle$; this case in turn formally reduces to the case $n=1$, which we deduce immediately from Lemma~\ref{L:induction on Null}.
\end{proof}

\begin{corollary}\label{C:Nullstellensatz property topologically of finite type}
Let $A$ be a commutative Banach Tate ring with the weak (resp.\ strong) Nullstellensatz property, and let $B$ be a commutative Banach $A$-algebra topologically of finite type.
Then, $B$ also has the weak (resp.\ strong) Nullstellensatz property, and for every maximal ideal $\frakm$ of $B$, $\frakm \cap A$ is a maximal ideal of $A$ and the ring homomorphism $A/(\frakm \cap A) \to B/\frakm$ is finite.
\end{corollary}
\begin{proof}
This follows from the proof of Theorem~\ref{T:Nullstellensatz property}.
\end{proof}

\begin{remark}
Let $A$ be a commutative Banach Tate ring.
Then $A$ has the weak Nullstellensatz property if and only if for any commutative Banach $A$-algebra $B$ topologically of finite type and for any maximal ideal $\frakm$ of $B$, the ring homomorphism $A/(\frakm \cap A) \to B/\frakm$ is finite.
\end{remark}

\begin{remark} \label{R:Jacobson-Tate}
Let $A$ be a strongly noetherian commutative Banach Tate ring.
From Corollary~\ref{C:Nullstellensatz property topologically of finite type}, we find that $A$ has the strong Nullstellensatz property if and only if it is a \emph{Jacobson-Tate ring} in the sense of \cite[Definition~3.1]{Lou17}. 
Consequently, \cite[Proposition~4.6]{Lou17} implies that for $K$ a discretely valued nonarchimedean field, any pseudoaffinoid algebra over $\frako_K$ has the strong Nullstellensatz property.
Moreover, a strongly noetherian commutative Banach Tate ring $A$ with the strong Nullstellensatz property is a Jacobson ring by \cite[Proposition~3.3]{Lou17}.
\end{remark}

\section{An application to Fargues--Fontaine curves}

It is evident from our definitions that any nonarchimedean field has the strong Nullstellensatz property. In this way, we recover from Theorem~\ref{T:Nullstellensatz property} the usual Nullstellensatz in rigid analytic geometry \cite[Corollary~6.1.2/3]{bgr}; however, this does not yield a totally independent derivation of that result because we used the main ingredient (Noether normalization) in the proof of Theorem~\ref{L:noether normalization}. 
More to the point is that we can identify some additional examples of rings with the strong Nullstellensatz property, as in the following examples arising from the geometry of Fargues--Fontaine curves.
Let us recall relevant rings defined in \cite{noeth}.

\begin{definition}
Let $p$ be a prime, and $q$ be a power of $p$, let $L$ be a perfect field containing $\FF_q$ which is complete with respect to the nontrivial multiplicative nonarchimedean norm $|-|$, let $E$ be a complete discretely valued field whose residue field is $\FF_q$, and fix a uniformizer $\varpi\in E$.
Let $\frako_L,\frako_E$ denote the valuation subrings of $L, E$.
Define the rings 
\[
A_{L,E}=W(\frako_L)_E[[\overline{x}]\colon \overline{x}\in L], \quad B_{L,E}=A_{L,E}\otimes_{\frako_E}E.
\]
For $t\in (0,+\infty)$, we define the multiplicative norm $\lambda_t \colon B_{L,E} \to \RR_{\geq 0}$ by the formula
\[
\lambda_t\left( \sum_{n\in \ZZ} \varpi^n [\overline{x}_n]\right)=\max\{p^{-n}|\overline{x}_n|^t\}.
\]
For $r>0$, we define $A_{L,E}^r$ as the completion of $A_{L,E}$ with respect to $\lambda_r$, and for $I=[s,r]\subset (0,+\infty)$, we define $B_{L,E}^I$ as the completion of $B_{L,E}$ with respect to the norm $\lambda_I=\max\{\lambda_s,\lambda_r\}$.
\end{definition}

\begin{remark}\label{R:rational localization}
We define the set $\Sigma_L$ as 
\[
\{r\in \RR_{>0} \mid \mbox{there is a rational number $m>0$ and $\overline{x}\in L^{\times}$ such that $r=-m/\log_p |\overline{x}|$}\}.
\]
Then for $r\in \Sigma_L$ as above, $A_{L,E}^r$ is the rational localization of $W(\frako_L)_E$ defined by $|p|^m \leq |[\overline{x}]|$.
Moreover, for $I=[s,r]\subset (0,r']\subset (0,+\infty)$ such that $r', r, s \in \Sigma_L$, $A_{L,E}^r$ and $B_{L,E}^I$ are rational localizations of $A_{L,E}^{r'}$.
\end{remark}

We adopt notation from \cite{FS24} for Fargues--Fontaine curves, except that the field $C$ in \cite{FS24} corresponds to our $L$.
Let $\hat{\overline{L}}$ be a completed algebraic closure of $L$, and let $G_L$ be the absolute Galois group of $L$.

\begin{definition}
We say that a point $x \in |\mathcal{Y}_L|$ \textit{corresponds to an untilt of a finite extension of $L$} if the natural map $\Spd \calH(x)\to \mathcal{Y}_L^{\diamond}=\Spd E\times \Spd L \to \Spd L$ of diamonds is finite \'{e}tale, where $\calH(x)$ is the completed residue field of $x$.
\end{definition}

\begin{lemma} \label{L:classical point}
Let $x \in |\mathcal{Y}_L|$ be a point corresponding to an untilt of a finite extension of $L$, and $\Spa(A,A^+)\subset \mathcal{Y}_L$ be an affinoid subspace containing $x$.
Then, the natural morphism $A\to \calH(x)$ is surjective.
In particular, for the maximal ideal $\frakm=\Ker(A\to \calH(x))$, $A/\frakm\cong \calH(x)$ is a nonarchimedean field. 
\end{lemma}
\begin{proof}
Let $L_1/L$ be a finite extension corresponding to the natural map $\Spd \calH(x)\to \Spd L$.
By \cite[Proposition~11.2.15]{FF18}, the natural map $\mathcal{Y}_{L_1}\to \mathcal{Y}_L$ is finite \'{e}tale.
Therefore, by replacing $L$ with $L_1$, we may assume that the map  $\Spd \calH(x)\to \Spd L$ is an isomorphism.
Then, the point $x$ provides the map $\Spd L\cong \Spd \calH(x) \to \Spd E\times \Spd L\to \Spd E$, which corresponds to an untilt $L^{\#}$ of $L$.
By \cite[Proposition II.1.4]{FS24}, we get a closed Cartier divisor $\Spa L^{\#}\to \mathcal{Y}_L$ whose image is $x$.
From this, we obtain $\Spa L^{\#}\cong \Spa \calH(x)$.
In particular, $\Spa \calH(x) \to \mathcal{Y}_L$ is a closed Cartier divisor, which proves the claim.
\end{proof}

\begin{lemma} \label{L:finite preimage}
A point $x \in |\mathcal{Y}_L|$ corresponds to an untilt of a finite extension of $L$ if and only if its preimage in $|\mathcal{Y}_{\hat{\overline{L}}}|$ is a finite set of classical points.
\end{lemma}
\begin{proof}
If $x$ corresponds to an untilt of a finite extension $L_1$ of $L$, then its preimage
consists of the classical points corresponding to the embeddings of $L_1$ into $\hat{\overline{L}}$.
Conversely, if the preimage of $x$ is finite, then the continuous action of $G_L$ on $|\mathcal{Y}_{\hat{\overline{L}}}|$ restricts to a continuous action on the preimage, which then becomes trivial on some closed-open subgroup of $G_L$. By replacing $L$ with the finite extension fixed by this closed-open subgroup, we reduce to the case where the preimage of $x$ is a single classical point. 
Since $\Spa \hat{\overline{L}} \to \Spa L$ is a v-covering, we may argue as in the first paragraph of the proof of \cite[Proposition~II.1.9]{FS24} to conclude that $x$ corresponds to an untilt of $L$ itself.
\end{proof}

\begin{lemma}\label{L:dim0}
For $r\notin \Sigma_L$, $B^{[r,r]}_{L,E}$ is a nonarchimedean field.
\end{lemma}
\begin{proof}
Since the topology of $B^{[r,r]}_{L,E}$ is defined by the multiplicative norm $\lambda_t$, it is enough to show that  $B^{[r,r]}_{L,E}$ is a field.
Let $x\in B^{[r,r]}_{L,E}\setminus\{0\}.$
Since $B^{[r,r]}_{L,E}$ is the completion of $B_{L,E}$ with respect to the norm $\lambda_t$, there is an element $y\in B_{L,E}$ such that $\lambda_t(x-y)<\lambda_t(y)$. 
The element $y$ can be written as $\sum_{n\geq n_0} \varpi^n[\overline{x}_n]$ for some integer $n_0$.
By replacing $y$, we may assume that this expansion is a finite sum.
Since $r\notin \Sigma_L$, $\lambda_t(\varpi^n[\overline{x}_n])$ are pairwise distinct.
Therefore, we may replace $y$ with $\varpi^n[\overline{x}_n]$ for some $n$.
Then $y$ is a unit in $B_{L,E}$.
Since $\lambda_t(x-y)<\lambda_t(y)$, $x=y+(x-y)$ is also a unit in $B^{[r,r]}_{L,E}$.
\end{proof}

\begin{theorem} \label{T:FF strong Null}
The rings $A^r_{L,E}$ and $B^{[s,r]}_{L,E}$ have the strong Nullstellensatz property.
Consequently, by Theorem~\ref{T:Nullstellensatz property}, any commutative Banach $A$-algebra topologically of finite type over $A^r_{L,E}$ and $B^{[s,r]}_{L,E}$ also has the strong Nullstellensatz property.
\end{theorem}
\begin{proof}
In the case of $B^{[s,r]}_{L,E}$ with $r = s \notin \Sigma_L$,
it is a nonarchimedean field by Lemma \ref{L:dim0}, so the strong Nullstellensatz property holds trivially. We exclude this case hereafter; the remaining cases correspond to subspaces of $|\mathcal{Y}_L|$ in which the classical points are dense.

Choose a complete algebraically closed extension $L'$ of $\hat{\overline{L}}$ such that $\Sigma_{L'} = (0, +\infty)$. 
Let $A^r_{L,E} \to B$ or $B^{[s,r]}_{L,E} \to B$
be a rational localization, let $B'$ be the base extension of $B$ from $L$ to $L'$,
and let $\frakm \subset B$ be a maximal ideal; then $\frakm \neq 0$ because the classical points of $\calM(B)$ are dense.
Choose any nonzero $x \in \frakm$,
let $S \subseteq \mathcal{Y}_L$ be the zero locus of $x$, 
and let $S' \subseteq \mathcal{Y}_{L'}$ be the preimage of $S$.
Since $\mathcal{Y}_{L'} \to \mathcal{Y}_L$ is a v-covering (as in the proof of Lemma~\ref{L:finite preimage}), $x$ remains nonzero in $B'$;
by the proof of \cite[Proposition~II.1.11]{FS24}, $S'$ is a finite set of classical points.
By \cite[Proposition~II.1.9]{FS24}, the preimage of $S$ in $\mathcal{Y}_{\hat{\overline{L}}}$
is a finite set of classical points; by Lemma~\ref{L:finite preimage},
$S$ is a finite set of points corresponding to untilts of finite extensions of $L$.
Therefore, the zero locus $T\subset \mathcal{Y}_L$ of $\frakm$ is also a finite set of points corresponding to untilts of finite extensions of $L$.
Since $T$ is connected, $T$ consists of a single point corresponding to an untilt of a finite extension of $L$.
This implies the claim at once by Lemma \ref{L:classical point}.
\end{proof}

\begin{remark}
Since the rings $A^r_{L,E}$ and $B^I_{L,E}$ are strongly noetherian \cite[Theorem~3.2, Theorem~4.10]{noeth}, Remark~\ref{R:Jacobson-Tate} and Theorem~\ref{T:FF strong Null} together imply that they are Jacobson-Tate rings, so the results of \cite{Lou17} apply.
\end{remark}

\begin{remark} \label{R:error}
Theorem~\ref{T:FF strong Null} can be used to correct an error in \cite{wear}: the given proof of \cite[Theorem~8.3]{wear} is incorrect, based on a similarly incorrect proof in \cite[Lemma~3.8]{kedlaya-liu-families}. As subsequent results in \cite{wear} (specifically \cite[Corollary~8.10, Proposition 9.10]{wear} depend only on \cite[Theorem~8.3]{wear} for the rings $A$ covered by Theorem~\ref{T:FF strong Null}, the latter may be used as a substitute.
\end{remark}

\begin{remark} \label{R:errata noeth}
We take the opportunity to record some errata for \cite{noeth}.
\begin{itemize}
\item
Hypothesis 2.1: ``contains $\mathbb{F}_q$'' should be ``equals $\mathbb{F}_q$''.
\item
Lemma 5.3(b): the displayed sequence should be
\[0\to L\to L' \to L'\otimes_L L' \to  L'\otimes_L L'\otimes_L L'.
\]
\item
Lemma 5.4: in the third paragraph of the proof, the equations $\lambda_1(v_l) \leq \epsilon^{-1} \lambda_1(q_l)$ and $\iota_1(v_l)-\iota_2(v_l) = q_l$ should read $\lambda_1(v_l) \leq \epsilon^{-1} \lambda_1(q_l-1)$ and $\lambda_1(\iota_1(v_l)-\iota_2(v_l) - q_l+1) \leq \epsilon \lambda_1(q_l-1)$. 
\item
Corollary 8.6: the deduction that the norm on $C/\mathfrak{m}$ is equivalent to $\beta$
requires knowing that $C/\mathfrak{m}$ is uniform, which has not yet been established.
In lieu of giving a direct correction, we observe that the claim follows from Theorem~\ref{T:FF strong Null} of this paper.
\end{itemize}
\end{remark}

\section{A final remark}

\begin{remark}
It would be of some interest to decide whether other natural classes of Banach Tate rings (or Huber Tate rings) admit the weak or strong Nullstellensatz property. One key case is when $A$ is a perfectoid ring over $\QQ_p$.
In this case, a theorem of Bhatt \cite[Theorem 1.12(2)]{bhatt-scholze} (see also \cite[Corollary~19.4.6]{prismatic})
implies that for every closed ideal $I$ of $A$, the uniform completion of $A/I$ is itself a quotient of $A$ by some possibly larger ideal.
When $I = \frakm$ is a maximal ideal, this implies that $A/\frakm$ is uniform, hence a perfectoid ring which is a field, and hence a perfectoid field by \cite[Theorem~4.2]{banach-field}. In particular, the topology on $A/\frakm$ is defined by some multiplicative norm.

The previous discussion implies that the weak and strong Nullstellensatz properties for perfectoid rings over $\QQ_p$ are equivalent, but it is not clear from this whether the weak property holds. It may however be possible to establish it for some more restricted class of perfectoid rings, such as those arising from He's criterion \cite{he}.
\end{remark}

\subsection*{Acknowledgement}
The first author was supported by NSF (grant DMS-2401536) and UC San Diego (Warschawski Professorship). The second author was supported by JSPS KAKENHI Grant Number JP23KJ0693. 

\subsection*{Data Availability Statement}
The authors declare that the manuscript has no associated data.

\subsection*{Conflict of Interest}
The authors declare that they have no conflict of interest.

\end{document}